\documentclass[12pt]{amsart}
\usepackage{amsfonts, amssymb, enumerate}
\usepackage{parskip, fullpage, verbatim}
\usepackage[colorlinks=true]{hyperref}
\usepackage{breakurl}
\usepackage{array}
\usepackage{longtable}

\newcommand\CA{{\mathcal A}} 
\newcommand\CB{{\mathcal B}}
\newcommand\CC{{\mathcal C}} 
\newcommand\CD{{\mathcal D}}

\newcommand\CIF{{\mathcal {IF}}} 
 
\newcommand\CIFAC{{\mathcal {IF\!AC}}}

\newcommand\R{{\varrho}}

\newcommand\BBK{{\mathbb K}}

\newcommand\BBR{{\mathbb R}}

\newcommand {\Sage}{\textsf{SAGE}}  

\newcommand\Der{{\operatorname{Der}}}

\newcommand\Poin{{\operatorname{Poin}}}

\newcommand\pdeg{\operatorname{pdeg}}

\renewcommand\th{{^{\text{th}}}}

\numberwithin{equation}{section}

\theoremstyle{plain}
\newtheorem{lemma}[equation]{Lemma}
\newtheorem{theorem}[equation]{Theorem}

\newtheorem{corollary}[equation]{Corollary}
\newtheorem{proposition}[equation]{Proposition}
\theoremstyle{definition}
\newtheorem{defn}[equation]{Definition}
\newtheorem{remark}[equation]{Remark}

\thanks{We acknowledge 
support from the DFG-priority program 
SPP1489 ``Algorithmic and Experimental Methods in
Algebra, Geometry, and Number Theory''.}


\begin{document}

%%%%%%%%%%%%%%%%%%%%%%%%%%%%%%%%%%%%%%%%%%%%%%%%%%%%%%%%%%%%%%%%%%%%%%
%%%%%%%%%%%%% top matter stuff
%%%%%%%%%%%%%%%%%%%%%%%%%%%%%%%%%%%%%%%%%%%%%%%%%%%%%%%%%%%%%%%%%%%%%%
\title[Nice restrictions of reflection arrangements]
{Nice restrictions of reflection arrangements}

\author[T. M\"oller]{Tilman M\"oller}
\address
{Fakult\"at f\"ur Mathematik,
Ruhr-Universit\"at Bochum,
D-44780 Bochum, Germany}
\email{tilman.moeller@rub.de}

\author[G. R\"ohrle]{Gerhard R\"ohrle}
\email{gerhard.roehrle@rub.de}


\allowdisplaybreaks

\begin{abstract}
In \cite{hogeroehrle:nice},
Hoge and the second author classified
all nice and all inductively factored
reflection arrangements.
In this note we extend this classification 
by determining all  
nice and all inductively factored
restrictions of 
reflection arrangements.
\end{abstract}

\maketitle

%%%%%%%%%%%%%%%%%%%%%%%%%%%%%%%%%%%%%%%%%%%%%%%%%%%%%%%%%%%%%%%%%%%%%%
%%%%%%%%%%%%% article body...
%%%%%%%%%%%%%%%%%%%%%%%%%%%%%%%%%%%%%%%%%%%%%%%%%%%%%%%%%%%%%%%%%%%%%%

%%%%%%%%%%%%%%%%%%%%%%%%%%%%%%%%%%%%%%%%%%%%%%%%%%%%%%%%%%%%%%%%%%%%%%
%%%%%%%%%%%%% \S1 Introduction
%%%%%%%%%%%%%%%%%%%%%%%%%%%%%%%%%%%%%%%%%%%%%%%%%%%%%%%%%%%%%%%%%%%%%%
\section{Introduction}

The notion of a nice arrangement is due to Terao \cite{terao:factored}.
This class generalizes the class of supersolvable arrangements, 
\cite{orliksolomonterao:hyperplanes}
(cf.\ \cite[Thm.\ 3.81]{orlikterao:arrangements}).
There is an inductive version of this notion, 
so called inductively factored arrangements,
see Definition \ref{def:indfactored}.
This inductive class (properly) contains the 
class of supersolvable arrangements and 
is (properly) contained in the class of 
inductively free arrangements, see
\cite[Rem.\ 3.33]{hogeroehrle:factored}.

For an overview on properties of nice 
and inductively factored arrangements, 
and for their connection with the 
Orlik-Solomon algebra, 
see \cite[\S 3]{orlikterao:arrangements}, 
\cite{jambuparis:factored}, 
and \cite{hogeroehrle:factored}.
In \cite{hogeroehrle:factored}, 
Hoge and the second author proved an 
addition-deletion theorem for nice arrangements, 
see Theorem \ref{thm:add-del-factored} below. 
This is an analogue of Terao's celebrated
addition-deletion theorem \ref{thm:add-del} for 
free arrangements for the class of 
nice arrangements.

In \cite{hogeroehrle:nice},
Hoge and the second author classified
all nice and all inductively factored
reflection arrangements.
Extending this earlier work, in
this note we classify all 
nice and all inductively factored restrictions 
$\CA^X$, for $\CA$ a reflection arrangement and $X$ in
the intersection lattice $L(\CA)$ of $\CA$, see Theorems 
\ref{thm:nice} and \ref{thm:indfac}. 
If $\CA^X$ is inductively factored for every $X \in L(\CA)$, 
then $\CA$ is called hereditarily inductively factored, see Definition
\ref{def:heredindfactored}.

In order to state our main results, we need 
a bit more notation: For fixed $r, \ell \geq 2$ 
and $0 \leq k \leq \ell$ we denote by
$\CA^k_\ell(r)$ the intermediate arrangements,
defined by Orlik and Solomon
in \cite[\S 2]{orliksolomon:unitaryreflectiongroups}
(see also \cite[\S 6.4]{orlikterao:arrangements}), that 
interpolate between the reflection arrangements $\CA(G(r,r,\ell)) =
\CA^0_\ell(r)$ and $\CA(G(r,1,\ell)) = \CA^\ell_\ell(r)$, 
of the monomial groups $G(r,r,\ell)$ and $G(r,1,\ell)$, respectively. 
The arrangements $\CA^k_\ell(r)$ are relevant for us, as they occur as 
restrictions of $\CA(G(r,r,\ell))$, 
\cite[Prop.\ 2.14]{orliksolomon:unitaryreflectiongroups}
(cf.\ \cite[Prop.~6.84]{orlikterao:arrangements}).
For $k \neq 0, \ell$, these are not reflection arrangements
themselves. See Section \ref{sec:akl} for further details.

Suppose that $W$ is a finite, unitary
reflection group acting on the complex 
vector space $V$.
Let $\CA(W) = (\CA(W),V)$ be the associated 
hyperplane arrangement of $W$.
We refer to $\CA(W)$ as a 
\emph{reflection arrangement}.
Thanks to Proposition \ref{prop:product-factored},
the question whether $\CA$ is
nice reduces to the case when $\CA$ is irreducible.
Therefore, we may assume that $W$ is irreducible.
First we recall the classification results from 
\cite{hogeroehrle:nice}

\begin{theorem}
[{\cite[Thm.\ 1.3, Thm.\ 1.5]{hogeroehrle:nice}}]
\label{thm:factoredrefl}
Let $W$ be a finite, irreducible, 
complex reflection group with reflection arrangement 
$\CA(W)$.
Then we have the following:
\begin{itemize}
\item[(i)]
$\CA(W)$ is nice if and only if 
either 
$\CA(W)$ is supersolvable or 
$W =  G(r,r,3)$ for $r \ge 3$.
\item[(ii)]
$\CA(W)$ is factored
if and only if $\CA(W)$ is hereditarily factored.
\end{itemize}
\end{theorem}

Thanks to Proposition \ref{prop:product-indfactored},
the question whether $\CA$ is
inductively factored reduces to the case when $\CA$ is irreducible.

\begin{theorem}
[{\cite[Cor.\ 1.4, Cor.\ 1.6]{hogeroehrle:nice}}]
\label{thm:indfactoredrefl}
Let $W$ be a finite, irreducible, 
complex reflection group with reflection arrangement 
$\CA(W)$.
Then we have the following:
\begin{itemize}
\item[(i)]
$\CA(W)$ 
is inductively factored if and only if it is supersolvable. 
\item[(ii)]
$\CA(W)$ is inductively  factored
if and only if $\CA(W)$ is hereditarily inductively factored.
\end{itemize}
\end{theorem}

Terao  \cite{terao:factored} showed that 
every supersolvable arrangement is factored.
Indeed, every supersolvable arrangement is inductively factored,
see Proposition \ref{prop:superindfactored}.
Moreover, Jambu and Paris showed that each inductively factored 
arrangement is inductively free, see Proposition \ref{prop:indfactoredindfree}.
Each of these classes of arrangements is properly contained in the other,
see \cite[Rem.\ 3.33]{hogeroehrle:factored}.

In view of these proper containments,
we first recall the classifications of 
the inductively free and 
the supersolvable restrictions of reflection arrangements, 
from 
\cite{amendhogeroehrle:indfree}
and \cite{amendhogeroehrle:super}, 
respectively, as they give an indication 
of the kind of results to be expected.
Here and later on we use the classification and 
labeling of the irreducible 
unitary reflection groups due to
Shephard and Todd, \cite{shephardtodd}. 

\begin{theorem}
[{\cite[Thm.\ 1.2]{amendhogeroehrle:indfree}}]
\label{thm:indfree-restriction}
Let $W$ be a finite, irreducible, complex 
reflection group with reflection arrangement 
$\CA = \CA(W)$ and let $X \in L(\CA)$. 
The restricted arrangement $\CA^X$ is inductively free 
if and only if one of the following holds:
\begin{itemize}
\item[(i)] 
$\CA$ is inductively free;
\item[(ii)] 
$W = G(r,r,\ell)$ and 
$\CA^X \cong \CA^k_p(r)$, where $p = \dim X$ and $p - 2 \leq k \leq p$; 
\item[(iii)] 
$W$ is one of $G_{24}, G_{27}, G_{29}, G_{31}, G_{33}$, or $G_{34}$ and $X \in L(\CA) \setminus \{V\}$ with  $\dim X \leq 3$.
\end{itemize}
\end{theorem}

Note that 
Proposition \ref{prop:superindfactored}
and Theorem \ref{thm:indfree-restriction}(iii)
imply that for $W$  an irreducible, complex 
reflection group of exceptional type,
for $W$ as in Theorem \ref{thm:indfree-restriction}(iii), for
$\CA = \CA(W)$ and $X \in L(\CA)$ 
with $\dim X \ge 4$, the restricted arrangement $\CA^X$ is 
not inductively factored. 

\begin{theorem}
[{\cite[Thm.\ 1.3]{amendhogeroehrle:super}}]
\label{thm:super-restriction}
Let $W$ be a finite, irreducible, complex 
reflection group with reflection arrangement 
$\CA = \CA(W)$ and let $X \in L(\CA)$ with $\dim X \ge 3$.
Then the restricted arrangement 
$\CA^X$ is supersolvable
if and only if 
one of the following holds: 
\begin{itemize}
\item[(i)] $\CA$ is supersolvable;
\item[(ii)] $W = G(r,r,\ell)$ and 
$\CA^X \cong \CA_p^p(r)$ or $\CA_p^{p-1}(r)$,
where $p = \dim X$; 
\item[(iii)] $\CA^X$ is $(E_6,A_3)$, $(E_7, D_4)$, $(E_7, A_2^2)$, or $(E_8, A_5)$.
\end{itemize}
\end{theorem}

In part (iii) of the theorem and later on
we use the convention 
to label the $W$-orbit 
of $X \in L(\CA)$ by the type $T$ which is 
the Shephard-Todd label \cite{shephardtodd}
of the complex reflection group $W_X$.
We then denote the restriction $\CA^X$ simply by the pair
$(W,T)$.

Note that thanks to 
Proposition \ref{prop:superindfactored}, 
every supersolvable restriction from 
Theorem \ref{thm:super-restriction}
is also inductively factored.

Thanks to the compatibility of 
nice arrangements and inductively factored arrangements  
with the product construction for arrangements, 
see Propositions 
\ref{prop:product-factored}
and
\ref{prop:product-indfactored},
as well as by the product rule \eqref{eq:restrproduct} for 
restrictions in products, 
the question whether the restrictions $\CA^X$
are nice or inductively factored 
reduces readily to the case when $\CA$ is irreducible.
Thus we may assume that $W$ is irreducible.
We can formulate our classification as follows:

\begin{theorem}
\label{thm:nice}
Let $W$ be a finite, irreducible, complex 
reflection group with reflection arrangement 
$\CA = \CA(W)$ and let $X \in L(\CA) \setminus \{V\}$. 
The restricted arrangement $\CA^X$ is 
nice if and only if one of the following holds:
\begin{itemize} 
\item[(i)]
$\CA^X$ is supersolvable;
\item[(ii)]
$\CA$ is nice; 
\item[(iii)]
$W = G(r,r,\ell)$, $\ell \ge 4$ and 
$\CA^X \cong \CA^{p-2}_p(r)$, where $p = \dim X$; 
\item[(iv)]
$\CA^X$ is one of $(E_6, A_1A_2), (E_7, A_4)$, or $(E_7, (A_1A_3)'')$.
\end{itemize}
\end{theorem}

Note that $(E_6, A_1A_2)$ and  $(E_7, A_4)$ in part (iv) above 
are isomorphic, see Lemma \ref{lem:isoms}(iv).

In contrast to the situation for the full reflection arrangements
(Theorems \ref{thm:factoredrefl} and \ref{thm:indfactoredrefl}),
the notions of niceness and inductive factoredness coincide
for their restricted counterparts.

\begin{theorem}
\label{thm:indfac}
Let $W$ be a finite, irreducible, complex 
reflection group with reflection arrangement 
$\CA = \CA(W)$ and let $X \in L(\CA) \setminus \{V\}$. 
The restricted arrangement $\CA^X$ is 
inductively factored if and only if it is factored.
\end{theorem}

We also extend both theorems to the corresponding hereditary subclasses.

\begin{theorem}
\label{thm:hered-indfac}
Let $W$ be a finite, irreducible, complex 
reflection group with reflection arrangement 
$\CA = \CA(W)$ and let $X \in L(\CA) \setminus \{V\}$. 
Then the restricted arrangement $\CA^X$ is 
(inductively) factored if and only if it is 
hereditarily (inductively) factored.
\end{theorem}

While Theorem \ref{thm:indfactoredrefl}(i) shows 
that the class of inductively factored reflection arrangements 
coincides with the class of supersolvable 
reflection arrangements, in contrast, Theorems
\ref{thm:super-restriction}, \ref{thm:nice}, 
and \ref{thm:indfac} show that 
the class of 
inductively factored restrictions of reflection arrangements 
properly contains the class consisting of 
supersolvable restrictions of reflection arrangements. 

The paper is organized as follows.
In the next section we recall the required notions 
and relevant properties of 
free, inductively free, 
supersolvable and nice arrangements 
mostly taken from
\cite{orlikterao:arrangements},
\cite{terao:factored} and 
\cite{hogeroehrle:factored}.

In Section \ref{sec:akl} we classify all nice and all 
inductively factored cases among the 
intermediate arrangements $\CA_\ell^k(r)$,
and complete the proofs of 
Theorems \ref{thm:nice} -- \ref{thm:hered-indfac}
in Section \ref{sec:proof}. 

For general information about arrangements 
we refer the reader to \cite{orlikterao:arrangements}.

\section{Recollections and Preliminaries}
\label{sect:prelims}

\subsection{Hyperplane arrangements}
\label{ssect:arrangements}
Let $\BBK$ be a field and let
$V = \BBK^\ell$ be an $\ell$-dimensional $\BBK$-vector space.
A \emph{(central) hyperplane arrangement} $\CA$ in $V$ 
is a finite collection of hyperplanes in $V$ each 
containing the origin of $V$.
We also use the term $\ell$-arrangement for $\CA$. 
The empty $\ell$-arrangement is denoted by $\Phi_\ell$. 

The \emph{lattice} $L(\CA)$ of $\CA$ is the set of subspaces of $V$ of
the form $H_1\cap \dotsm \cap H_i$ where $\{ H_1, \ldots, H_i\}$ is a subset
of $\CA$. 
For $X \in L(\CA)$, we have two associated arrangements, 
firstly
$\CA_X :=\{H \in \CA \mid X \subseteq H\} \subseteq \CA$,
the \emph{localization of $\CA$ at $X$}, 
and secondly, 
the \emph{restriction of $\CA$ to $X$}, $(\CA^X,X)$, where 
$\CA^X := \{ X \cap H \mid H \in \CA \setminus \CA_X\}$.
Note that $V$ belongs to $L(\CA)$
as the intersection of the empty 
collection of hyperplanes and $\CA^V = \CA$. 
The lattice $L(\CA)$ is a partially ordered set by reverse inclusion:
$X \le Y$ provided $Y \subseteq X$ for $X,Y \in L(\CA)$.

If $0 \in H$ for each $H$ in $\CA$, then 
$\CA$ is called \emph{central}.
If $\CA$ is central, then the \emph{center} 
$T_\CA := \cap_{H \in \CA} H$ of $\CA$ is the unique
maximal element in $L(\CA)$  with respect
to the partial order.
Throughout, we only consider central arrangements.

The \emph{product}
$\CA = (\CA_1 \times \CA_2, V_1 \oplus V_2)$ 
of two arrangements $(\CA_1, V_1), (\CA_2, V_2)$
is defined by
\begin{equation}
\label{eq:product}
\CA := \CA_1 \times \CA_2 = \{H_1 \oplus V_2 \mid H_1 \in \CA_1\} \cup 
\{V_1 \oplus H_2 \mid H_2 \in \CA_2\},
\end{equation}
see \cite[Def.~2.13]{orlikterao:arrangements}.

An arrangement $\CA$ is called \emph{reducible},
if it is of the form $\CA = \CA_1 \times \CA_2$, where 
$\CA_i \ne \Phi_0$ for $i=1,2$, else $\CA$
is \emph{irreducible}, 
\cite[Def.~2.15]{orlikterao:arrangements}.

If $\CA = \CA_1 \times \CA_2$ is a product, 
then by \cite[Prop.~2.14]{orlikterao:arrangements}
there is a lattice isomorphism
\[
 L(\CA_1) \times L(\CA_2) \cong L(\CA) \quad \text{by} \quad
(X_1, X_2) \mapsto X_1 \oplus X_2.
\]
With \eqref{eq:product}, it is easy to see that
for $X =  X_1 \oplus X_2 \in L(\CA)$, we have 
$\CA _X =  ({\CA_1})_{X_1} {\times} ({\CA_2})_{X_2}$
and 
\begin{equation}
\label{eq:restrproduct}
\CA^X = \CA_1^{X_1} \times \CA_2^{X_2}.
\end{equation}

\subsection{Free hyperplane arrangements}
\label{ssect:free}
Let $S = S(V^*)$ be the symmetric algebra of the dual space $V^*$ of $V$.
Let $\Der(S)$ be the $S$-module of $\BBK$-derivations of $S$.
Since $S$ is graded, 
$\Der(S)$ is a graded $S$-module.

Let $\CA$ be an arrangement in $V$. 
Then for $H \in \CA$ we fix $\alpha_H \in V^*$ with
$H = \ker \alpha_H$.
The \emph{defining polynomial} $Q(\CA)$ of $\CA$ is given by 
$Q(\CA) := \prod_{H \in \CA} \alpha_H \in S$.
The \emph{module of $\CA$-derivations} of $\CA$ is 
defined by 
\[
D(\CA) := \{\theta \in \Der(S) \mid \theta(Q(\CA)) \in Q(\CA) S\} .
\]
We say that $\CA$ is \emph{free} if 
$D(\CA)$ is a free $S$-module, cf.\ \cite[\S 4]{orlikterao:arrangements}.

If $\CA$ is a free arrangement, then the $S$-module
$D(\CA)$ admits a basis of $n$ homogeneous derivations, 
say $\theta_1, \ldots, \theta_n$, \cite[Prop.\ 4.18]{orlikterao:arrangements}.
While the $\theta_i$'s are not unique, their polynomial 
degrees $\pdeg \theta_i$ 
are unique (up to ordering). This multiset is the set of 
\emph{exponents} of the free arrangement $\CA$
and is denoted by $\exp \CA$.

Terao's celebrated \emph{Addition-Deletion Theorem} 
which we recall next
plays a 
pivotal role in the study of free arrangements, 
\cite[\S 4]{orlikterao:arrangements}.
For $\CA$ non-empty, 
let $H_0 \in \CA$.
Define $\CA' := \CA \setminus\{ H_0\}$,
and $\CA'' := \CA^{H_0} = \{ H_0 \cap H \mid H \in \CA'\}$,
the restriction of $\CA$ to $H_0$.
Then $(\CA, \CA', \CA'')$ is a \emph{triple} of arrangements,
\cite[Def.\ 1.14]{orlikterao:arrangements}. 

\begin{theorem}[\cite{terao:freeI}]
\label{thm:add-del}
Suppose that $\CA$ is a non-empty $\ell$-arrangement.
Let  $(\CA, \CA', \CA'')$ be a triple of arrangements. Then any 
two of the following statements imply the third:
\begin{itemize}
\item[(i)] $\CA$ is free with $\exp \CA = \{ b_1, \ldots , b_{\ell -1}, b_\ell\}$;
\item[(ii)] $\CA'$ is free with $\exp \CA' = \{ b_1, \ldots , b_{\ell -1}, b_\ell-1\}$;
\item[(iii)] $\CA''$ is free with $\exp \CA'' = \{ b_1, \ldots , b_{\ell -1}\}$.
\end{itemize}
\end{theorem}
There are various stronger notions of freeness
which we discuss in the following subsections.

\subsection{Inductively free arrangements}
\label{ssect:indfree}

Theorem \ref{thm:add-del}
motivates the notion of 
\emph{inductively free} arrangements,  see 
\cite{terao:freeI} or 
\cite[Def.\ 4.53]{orlikterao:arrangements}.

\begin{defn}
\label{def:indfree}
The class $\CIF$ of \emph{inductively free} arrangements 
is the smallest class of arrangements subject to
\begin{itemize}
\item[(i)] $\Phi_\ell \in \CIF$ for each $\ell \ge 0$;
\item[(ii)] if there exists a hyperplane $H_0 \in \CA$ such that both
$\CA'$ and $\CA''$ belong to $\CIF$, and $\exp \CA '' \subseteq \exp \CA'$, 
then $\CA$ also belongs to $\CIF$.
\end{itemize}
\end{defn}

\subsection{Supersolvable arrangements}
\label{ssect:supersolve}

Let $\CA$ be an arrangement. We say
that $X \in L(\CA)$ is \emph{modular}
provided $X + Y \in L(\CA)$ for every $Y \in L(\CA)$,
 \cite[Cor.\ 2.26]{orlikterao:arrangements}.

\begin{defn}
[{\cite{stanley:super}}]
\label{def:super}
Let $\CA$ be a central arrangement of rank $r$.
We say that $\CA$ is 
\emph{supersolvable} 
provided there is a maximal chain
\[
V = X_0 < X_1 < \ldots < X_{r-1} < X_r = T_\CA
\]
 of modular elements $X_i$ in $L(\CA)$,
cf.\ \cite[Def.\ 2.32]{orlikterao:arrangements}.
\end{defn}

The connection of this notion with freeness is
due to Jambu and Terao.

\begin{theorem}
[{\cite[Thm.\ 4.2]{jambuterao:free}}]
\label{thm:superindfree}
A supersolvable arrangement is inductively free.
\end{theorem}

\subsection{Nice and inductively factored arrangements}
\label{ssect:factored}

The notion of a \emph{nice} or \emph{factored} 
arrangement goes back to Terao \cite{terao:factored}.
It generalizes the concept of a supersolvable arrangement, see
\cite[Thm.\ 5.3]{orliksolomonterao:hyperplanes} and 
\cite[Prop.\ 2.67, Thm.\ 3.81]{orlikterao:arrangements}.
Terao's main motivation was to give a 
general combinatorial framework to 
deduce factorizations of the underlying Orlik-Solomon algebra,
see also \cite[\S 3.3]{orlikterao:arrangements}.
We recall the relevant notions  
from \cite{terao:factored}
(cf.\  \cite[\S 2.3]{orlikterao:arrangements}):

\begin{defn}
\label{def:factored}
Let $\pi = (\pi_1, \ldots , \pi_s)$ be a partition of $\CA$.
\begin{itemize}
\item[(a)]
$\pi$ is called \emph{independent}, provided 
for any choice $H_i \in \pi_i$ for $1 \le i \le s$,
the resulting $s$ hyperplanes are linearly independent, i.e.\
$r(H_1 \cap \ldots \cap H_s) = s$.
\item[(b)]
Let $X \in L(\CA)$.
The \emph{induced partition} $\pi_X$ of $\CA_X$ is given by the non-empty 
blocks of the form $\pi_i \cap \CA_X$.
\item[(c)]
$\pi$ is
\emph{nice} for $\CA$ or a \emph{factorization} of $\CA$  provided 
\begin{itemize}
\item[(i)] $\pi$ is independent, and 
\item[(ii)] for each $X \in L(\CA) \setminus \{V\}$, 
the induced partition $\pi_X$ admits a block 
which is a singleton. 
\end{itemize}
\end{itemize}
If $\CA$ admits a factorization, then we also say that $\CA$ is \emph{factored} or \emph{nice}.
\end{defn}

\begin{remark}
\label{rem:factored}
The class of nice arrangements is closed under taking localizations.
For, if $\CA$ is non-empty and   
$\pi$ is a nice partition of $\CA$, then the non-empty parts of the 
induced partition $\pi_X$ form a nice partition of $\CA_X$
for each $X \in L(\CA)\setminus\{V\}$;
cf.~the proof of \cite[Cor.\ 2.11]{terao:factored}.
\end{remark}

The main motivation in \cite{terao:factored} 
to introduce 
the notion of a nice or factored partition was that it allows for a 
combinatorial characterization of tensor factorizations
as a graded $\BBK$-algebra of the Orlik-Solomon algebra of an 
arrangement.
We record a set of consequences of this result that are 
relevant for our purposes, see \cite{terao:factored}   
(cf.\  \cite[\S 3.3]{orlikterao:arrangements}).

\begin{corollary}
\label{cor:teraofactored}
Let  $\pi = (\pi_1, \ldots, \pi_s)$ be a factorization of $\CA$.
Then the following hold:
\begin{itemize}
\item[(i)] $s = r = r(\CA)$ and 
\[
\Poin(\CA,t) = \prod_{i=1}^r (1 + |\pi_i|t);
\]
\item[(ii)]
the multiset $\{|\pi_1|, \ldots, |\pi_r|\}$ only depends on $\CA$;
\item[(iii)]
for any $X \in L(\CA)$, we have
\[
r(X) = |\{ i \mid \pi_i \cap \CA_X \ne \varnothing \}|.
\]  
\end{itemize}
\end{corollary}

\begin{remark}
\label{rem:exponents}
Suppose that $\CA$ is free of rank $r$.
Then $\CA = \Phi_{\ell-r} \times \CA_0$, 
where $\CA_0$ is an essential, free $r$-arrangement
(cf.\ \cite[\S 3.2]{orlikterao:arrangements}), and so,  
$\exp \CA = \{0^{\ell-r}, \exp \CA_0\}$.
Suppose that $\pi = (\pi_1, \ldots, \pi_r)$ is a nice partition 
of $\CA$. 
Then by the factorization properties of the Poincar\'e polynomials
for free and factored arrangements, we have 
\begin{equation*}
\label{eq:exp}
\exp \CA = \{0^{\ell-r}, |\pi_1|, \ldots, |\pi_r|\}.
\end{equation*}
In particular, if $\CA$ is essential, then 
\begin{equation*}
\label{eq:ess-exp}
\exp \CA = \{|\pi_1|, \ldots, |\pi_\ell|\}.
\end{equation*}
\end{remark}

Finally, we record 
\cite[Ex.\ 2.4]{terao:factored},
which shows that nice arrangements generalize 
supersolvable ones 
(cf.\  \cite[Thm.\ 5.3]{orliksolomonterao:hyperplanes}, 
\cite[Prop.\ 3.2.2]{jambu:factored}, 
\cite[Prop.\ 2.67, Thm.\ 3.81]{orlikterao:arrangements}).

\begin{proposition}
\label{prop:ssfactored}
Let $\CA$ be a central,  supersolvable arrangement of rank $r$.
Let 
\[
V = X_0 < X_1 < \ldots < X_{r-1} < X_r = T_\CA
\]
be a  maximal chain of modular elements in $L(\CA)$.
Define
$\pi_i = \CA_{X_i} \setminus \CA_{X_{i-1}}$
for $1 \le i \le r$.
Then $\pi = (\pi_1, \ldots, \pi_r)$ is a nice partition of $\CA$.
\end{proposition}

In \cite[Prop.\ 3.29]{hogeroehrle:factored},
it was shown that
the product construction behaves well with factorizations.

\begin{proposition}
\label{prop:product-factored}
Let $\CA_1, \CA_2$ be two arrangements.
Then  $\CA = \CA_1 \times \CA_2$ is nice
if and only if both 
$\CA_1$ and $\CA_2$ are nice.
\end{proposition}

Following Jambu and Paris 
\cite{jambuparis:factored}, 
we introduce further notation.
Suppose $\CA$ is not empty. 
Let $\pi = (\pi_1, \ldots, \pi_s)$ be a partition of $\CA$.
Let $H_0 \in \pi_1$ and 
let $(\CA, \CA', \CA'')$ be the triple associated with $H_0$. 
Then $\pi$ induces a 
partition 
$\pi'$ of 
$\CA'$, i.e.\ the non-empty 
subsets $\pi_i \cap \CA'$.
Note that since $H_0 \in \pi_1$, we have
$\pi_i \cap \CA' = \pi_i$ 
for $i = 2, \ldots, s$. 
Also, associated with $\pi$ and $H_0$, we define 
the \emph{restriction map}
\[
\R := \R_{\pi,H_0} : \CA \setminus \pi_1 \to \CA''\ \text{ given by } \ H \mapsto H \cap H_0
\]
and set 
\[
\pi_i'' := \R(\pi_i) = \{H \cap H_0 \mid H \in \pi_i\} \
\text{ for }\  2 \le i \le s.
\]
In general, $\R$ need not be surjective nor injective.
However, since we are only concerned with cases when 
$\pi'' = (\pi_2'', \ldots, \pi_s'')$ is a
partition of $\CA''$,  
$\R$ has to be onto and 
$\R(\pi_i) \cap \R(\pi_j) = \varnothing$ for $i \ne j$.

The following 
analogue of Terao's 
Addition-Deletion Theorem \ref{thm:add-del} for 
free arrangements for the class of 
nice arrangements is proved in 
\cite[Thm.\ 3.5]{hogeroehrle:factored}.

\begin{theorem}
\label{thm:add-del-factored}
Suppose that $\CA \ne \Phi_\ell$.
Let $\pi = (\pi_1, \ldots, \pi_s)$ be a  partition  of $\CA$.
Let $H_0 \in \pi_1$ and 
let $(\CA, \CA', \CA'')$ be the triple associated with $H_0$. 
Then any two of the following statements imply the third:
\begin{itemize}
\item[(i)] $\pi$ is nice for $\CA$;
\item[(ii)] $\pi'$ is nice for $\CA'$;
\item[(iii)] $\R: \CA \setminus \pi_1 \to \CA''$ 
is bijective and $\pi''$ is nice for $\CA''$.
\end{itemize}
\end{theorem}

Note the bijectivity condition on $\R$ 
in Theorem \ref{thm:add-del-factored}
is necessary, 
cf.~\cite[Ex.\ 3.3]{hogeroehrle:factored}.
Theorem \ref{thm:add-del-factored} 
motivates
the following stronger notion of factorization, 
cf.\ \cite{jambuparis:factored}, \cite[Def.\ 3.8]{hogeroehrle:factored}.

\begin{defn} 
\label{def:indfactored}
The class $\CIFAC$ of \emph{inductively factored} arrangements 
is the smallest class of pairs $(\CA, \pi)$ of 
arrangements $\CA$ together with a partition $\pi$
subject to
\begin{itemize}
\item[(i)] 
$(\Phi_\ell, (\varnothing)) \in \CIFAC$ for each $\ell \ge 0$;
\item[(ii)] 
if there exists a partition $\pi$ of $\CA$ 
and a hyperplane $H_0 \in \pi_1$ such that 
for the triple $(\CA, \CA', \CA'')$ associated with $H_0$ 
the restriction map $\R = \R_{\pi, H_0} : \CA \setminus \pi_1 \to \CA''$ 
is bijective and for the induced partitions $\pi'$ of $\CA'$ and 
$\pi''$ of $\CA''$ 
both $(\CA', \pi')$ and $(\CA'', \pi'')$ belong to $\CIFAC$, 
then $(\CA, \pi)$ also belongs to $\CIFAC$.
\end{itemize}
If $(\CA, \pi)$ is in $\CIFAC$, then we say that
$\CA$ is \emph{inductively factored with respect to $\pi$}, or else
that $\pi$ is an \emph{inductive factorization} of $\CA$. 
Frequently, we simply say $\CA$ is \emph{inductively factored} without 
reference to a specific inductive factorization of $\CA$.
\end{defn}

In \cite[Prop.\ 3.30]{hogeroehrle:factored},
Proposition \ref{prop:product-factored}
was strengthened  
further by showing that 
the compatibility with products restricts 
to the class of inductively factored arrangements.

\begin{proposition}
\label{prop:product-indfactored}
Let $\CA_1, \CA_2$ be two arrangements.
Then  $\CA = \CA_1 \times \CA_2$ is 
inductively factored if and only if both 
$\CA_1$ and $\CA_2$ are 
inductively factored  and in that case
the multiset of exponents of $\CA$ is given by 
$\exp \CA = \{\exp \CA_1, \exp \CA_2\}$.
\end{proposition}

The connection with the previous notions is as follows.

\begin{proposition}
[{\cite[Prop.\ 3.11]{hogeroehrle:factored}}]
\label{prop:superindfactored}
If $\CA$ is supersolvable, then $\CA$ is inductively factored.
\end{proposition}

\begin{proposition}
[{\cite[Prop.\ 2.2]{jambuparis:factored}, 
\cite[Prop.\ 3.14]{hogeroehrle:factored}}]
\label{prop:indfactoredindfree}
Let $\pi = (\pi_1, \ldots, \pi_r)$ be an inductive factorization of $\CA$. 
Then $\CA$ is inductively free with exponents 
$\exp \CA = \{0^{\ell-r}, |\pi_1|, \ldots, |\pi_r|\}$.
\end{proposition}

\begin{remark}
\label{rem:indfactable}
In analogy to inductively free arrangements, for
inductively factored arrangements one 
can present a so called induction table of factorizations, 
cf.~\cite[Rem.\ 3.16]{hogeroehrle:factored}.

If $\CA$ is inductively factored, then  $\CA$ is inductively free,
by Proposition \ref{prop:indfactoredindfree}.
The latter can be described by a so called 
\emph{induction table}, cf.~\cite[\S 4.3, p.~119]{orlikterao:arrangements}.
In this process we start with an inductively free arrangement
and add hyperplanes successively ensuring that 
part (ii) of Definition \ref{def:indfree} is satisfied.
This process is referred to as \emph{induction of hyperplanes}.
This procedure amounts to 
choosing a total order on $\CA$, say 
$\CA = \{H_1, \ldots, H_n\}$, 
so that each of the subarrangements 
$\CA_0 := \Phi_\ell$, $\CA_i := \{H_1, \ldots, H_i\}$
and each of the restrictions $\CA_i^{H_i}$ is inductively free
for $i = 1, \ldots, n$.
In the associated induction table we record in the $i$-th row the information 
of the $i$-th step of this process, by 
listing $\exp \CA_i' = \exp \CA_{i-1}$, 
the defining form $\alpha_{H_i}$ of $H_i$, 
as well as $\exp \CA_i'' = \exp \CA_i^{H_i}$, 
for $i = 1, \ldots, n$.

The proof of Proposition \ref{prop:indfactoredindfree} 
shows 
that if $\pi$ is an inductive factorization of $\CA$
and $H_0 \in \CA$ is distinguished with respect to $\pi$, then 
the triple $(\CA, \CA', \CA'')$ with respect to $H_0$ 
is a triple of inductively free arrangements.
Thus an induction table of $\CA$ can be constructed,
compatible with suitable inductive factorizations of 
the subarrangements $\CA_i$.

Let $\CA = \{H_1, \ldots, H_n\}$ be a 
choice of a total order on $\CA$.
Then, starting with the empty partition for $\Phi_\ell$, we can attempt to
build inductive factorizations $\pi_i$ of $\CA_i$
consecutively, resulting in an inductive factorization 
$\pi = \pi_n$ of $\CA = \CA_n$.
This is achieved by invoking Theorem \ref{thm:add-del-factored}
repeatedly in order to derive that each $\pi_i$ is an inductive factorization of $\CA_i$.

We then add the inductive factorizations $\pi_i$
of $\CA_i$ as additional data into an induction table for $\CA$
(or else record to which part of $\pi_{i-1}$ the new hyperplane $H_i$ is appended to).
The data in such an extended induction table 
together with the ``Addition'' part of Theorem \ref{thm:add-del-factored}
then proves that $\CA$ is 
inductively factored. 
We refer to this technique as \emph{induction of factorizations} and the
corresponding table as an \emph{induction table of factorizations} for $\CA$.
See Table \ref{table0} for an example.
\end{remark}

\begin{defn}
\label{def:heredindfactored}
In analogy to hereditary freeness and hereditary inductive freeness, 
\cite[Def.\ 4.140, p.\ 253]{orlikterao:arrangements}, 
we say that $\CA$ is \emph{hereditarily factored}
provided $\CA^X$ is factored for every $X \in L(\CA)$
and that $\CA$ is \emph{hereditarily inductively factored}
provided $\CA^X$ is inductively factored for every $X \in L(\CA)$.
\end{defn}

We recall a useful fact from \cite[Lem.\ 3.27]{hogeroehrle:factored}.

\begin{lemma}
\label{lem:3-arr}
Suppose that $\ell = 3$. Then 
$\CA$ is (inductively) factored if and only if it is 
hereditarily (inductively) factored.
\end{lemma}

By Remark \ref{rem:factored}, the class of 
factored arrangements
is closed under taking localizations.
This feature descends to the class of 
inductively factored arrangements
and its hereditary subclass.

\begin{theorem}
[{\cite[Thm.\ 1.1, Rem.\ 3.7]{moellerroehrle:factored}}]
\label{thm:factored-localization}
The class of (hereditarily) inductively factored arrangements
is closed under taking localizations.
\end{theorem}

\section{The intermediate arrangements $\CA_\ell^k(r)$}
\label{sec:akl}

In this section we discuss the intermediate 
arrangements $\CA^k_\ell(r)$ from 
\cite[\S 2]{orliksolomon:unitaryreflectiongroups}
(cf.\ \cite[\S 6.4]{orlikterao:arrangements}) 
in more detail, as they occur as 
restrictions of $\CA(G(r,r,\ell))$, 
\cite[Prop.\ 2.14]{orliksolomon:unitaryreflectiongroups}
(cf.\ \cite[Prop.~6.84]{orlikterao:arrangements}).
They interpolate between the
reflection arrangements of $G(r,r,\ell)$ and $G(r,1,\ell)$. 
For  $\ell \geq 2$ and $0 \leq k \leq \ell$ the defining polynomial of
$\CA^k_\ell(r)$ is given by
\[
Q(\CA^k_\ell(r)) = x_1 \cdots x_k\prod\limits_{\substack{1 \leq i < j \leq \ell\\ 0 \leq n < r}}(x_i - \zeta^nx_j),
\]
where $\zeta$ is a primitive $r\th$ root of unity,
so that 
$\CA^\ell_\ell(r) = \CA(G(r,1,\ell))$ and 
$\CA^0_\ell(r) = \CA(G(r,r,\ell))$. 
Note that for $0 < k < \ell$, 
$\CA^k_\ell(r)$ is not a reflection arrangement.
Thanks to 
\cite[Props.\ 2.11, 2.13]{orliksolomon:unitaryreflectiongroups},
each of these arrangements is free with 
\[
\exp(\CA_\ell^k(r)) = \{1, r+1, 2r+1, \dots, (\ell-2)r+1, (\ell-1)r+k-\ell+1\}
\]
(cf.~\cite[Props.~6.82, 6.85]{orlikterao:arrangements}).
The supersolvable and inductively free instances among the $\CA^k_\ell(r)$ 
are classified in 
Theorems \ref{thm:indfree-restriction}(ii) and \ref{thm:super-restriction}(ii). 

We abbreviate the hyperplanes in $\CA_{\ell}^k(r)$ as follows. 
For $1\leq a < b \leq \ell, 0\leq n < r$ and 
$1 \leq c \leq k$, let
\[
H_{a,b}^n:=\ker(x_a-\zeta^n x_b) \quad\text{ and }\quad H_c := \ker (x_c). 
\] 

\begin{lemma}
\label{lem:al-3} 
Let $\CA = \CA^{\ell-3}_\ell(r)$ for $r \ge 2, \ell \ge 4$.
Then $\CA$ is not nice.
\end{lemma}

\begin{proof}
It suffices to show the result for $\CA^1_4(r)$.
For, let $\CA = \CA^{\ell-3}_\ell(r)$ for $r \ge 2, \ell \ge 5$.
Then $\CA^1_4(r)$ is realized as a localization 
of $\CA$ as follows. Setting 
\[
X := \bigcap\limits_{\substack{\ell-3 \leq a < b \leq \ell\\ 0 \leq n < r}} H_{a,b}^n
\]
one readily checks that 
\[
\CA_X \cong \CA^1_4(r).
\]
Consequently, if $\CA^1_4(r)$ is not nice, 
then neither is $\CA$, thanks to 
Remark \ref{rem:factored}. 

Next we show that $\CA = \CA^1_4(r)$ fails to be nice. 
We are going to use Corollary \ref{cor:teraofactored} 
repeatedly on lattice elements of rank 2, 
to show that no partition of $\CA$ can be nice.
For that suppose $\pi = (\pi_1, \pi_2, \pi_3, \pi_4)$ 
is a nice partition of $\CA$. 
Since $\CA$ is free with $\exp(\CA)=\{1,r+1,2r+1,3r-2\}$, 
the cardinalities of the parts of $\pi$ coincide with the exponents of $\CA$,
by Remark \ref{rem:exponents}. 

First note that for $\{a,b,c,d\}=\{1,2,3,4\}$ 
and any $0\leq m,n<r$, 
we have $\{H_{a,b}^m,H_{c,d}^n\}\in L(\CA)$ 
of rank 2. 
So by Corollary \ref{cor:teraofactored}(iii), 
there can't be hyperplanes of this kind 
in the same $\pi_i$ unless they have a coordinate in common.
Since there are 6 different pairs of distinct 
coordinates $(a,b)$ for $H_{a,b}$ 
and each such excludes exactly one other such pair, 
there can be at most 3 different pairs of coordinates showing up among the 
members within the same $\pi_i$.
Likewise, the coordinate hyperplane $H_1$ 
can't be in the same part as any of the $H_{a,b}$ with $a\neq 1$.

If two hyperplanes have only one coordinate in common, 
then their intersection is contained in a third hyperplane.
That is
\begin{equation}
\label{eq:delta}
\{H_{a,b}^m,H_{a,c}^n,H_{b,c}^p\}\in L(\CA)
\text{ for } 
1\leq a<b<c\leq 4,\hspace{5pt} 0\leq m,n < r 
\text{ and } p \equiv n-m \!\!\! \mod r.
\end{equation}

The intersection of two hyperplanes 
with the same pair of coordinates $(a,b)$ 
is contained in every hyperplane associated with the 
pair $(a,b)$, and if $a=1$, is also contained in $H_1$,
of course. 
Thus, by Corollary \ref{cor:teraofactored}(iii) again,
each of the two sets
$\{H_{1,b}^z\mid 0\leq z < r\}\cup \{H_1\}$ 
and $\{H_{a,b}^z \mid 0\leq z < r\}$ for $a\neq 1$ 
has to split into two parts of $\pi$.

Now let $|\pi_1| = 3r-2$. 
Since $|\{H_{a,b}^z\mid 0\leq z < r\}|=r$ 
for any $1\leq a<b\leq 4$ 
and because of the previous restrictions discussed above, 
if $r\geq 3$, then $\pi_1$ has to contain hyperplanes of 
three different coordinate pairs.
In case $r=2$, it is sufficient to consider the case 
$\pi_1=\{H_{1,2}^0, H_{1,2}^1,H_{1,3}^0, H_{1,3}^1\}$.
The other possibilities are equivalent to this one. 
Suppose that $H_{3,4}^0\in \pi_2$ and $H_{3,4}^1\in \pi_3$, 
since they have to be in separate parts. 
By \eqref{eq:delta}, we have 
$H_{1,4}^0 \in \pi_2\cap \pi_3$, which is absurd. 
So let $r\geq 2$ and suppose that $\pi_1$ contains hyperplanes 
of three different coordinate pairs.
The possibilities are as follows:
\begin{enumerate}[(i)]
\item $H_{1,2},H_{1,3},H_{1,4} \in \pi_1$;
\item $H_{1,2},H_{2,3},H_{2,4} \in \pi_1$;
\item $H_{1,3},H_{2,3},H_{3,4} \in \pi_1$;
\item $H_{1,4},H_{2,4},H_{3,4} \in \pi_1$.
\end{enumerate}
The cases (ii), (iii) and (iv) are equivalent, by symmetry, 
and only in case (i) we might have $H_1\in \pi_1$.

First consider case (ii). 
Since $|\pi_1|=3r-2$ 
there are exactly two hyperplanes 
with the same set of coordinates that are not in $\pi_1$. 
More precisely, this must be one with pair $(2,3)$ and one 
with the pair $(2,4)$. 
Again, because of symmetry, it is sufficient to 
consider the case when $H_{2,3}^0, H_{2,4}^0 \not\in \pi_1$. 
Suppose that $H_{2,3}^0\in \pi_2$. 
It follows from \eqref{eq:delta} that 
$H_{1,3}^0, H_{3,4}^1\in \pi_2$. 
Then it further follows that  
$H_{2,4}^0\in \pi_2$. But 
this is a contradiction, 
as two hyperplanes in $\pi_2$ 
without a common coordinate are disallowed. 

Now consider case (i). 
Independent whether or not $H_1$ 
belongs to $\pi_1$, 
there are at least two hyperplanes 
of the form as in (i) that do not belong to $\pi_1$. 
Because of symmetry, it suffices to consider the case
$H_{1,2}^0, H_{1,3}^0 \not \in \pi_1$.
Suppose $H_{1,2}^0\in \pi_2$. 
It follows from \eqref{eq:delta} that 
$H_{2,3}^1\in \pi_2$. It then follows that also 
$H_{1,3}^0 \in \pi_2$. 
However, since there must be a hyperplane 
$H_{1,4}^k$ in $\pi_1$, 
we also get $H_{2,4}^k\in \pi_2$, 
which is a contradiction, again because
having two hyperplanes in $\pi_2$ 
without a common coordinate is not possible. 

Ultimately, we see that there is no nice partition of 
$\CA^1_4(r)$.
\end{proof}

\begin{lemma}
\label{lem:al-2}
Let $\CA = \CA^{\ell-2}_\ell(r)$ for $r, \ell \ge 2$.
Then $\CA$ is inductively factored.
\end{lemma}

\begin{proof}
The idea is to start with $\CA_{\ell-1}^{\ell-2}(r) \times \Phi_1$ 
and give an inductive chain of factorizations 
up to $\CA^{\ell-2}_\ell(r)$.
By \cite[Lem.\ 3.1]{amendhogeroehrle:super}, 
$\CA_{\ell-1}^{\ell-2}(r)$ is supersolvable, 
and therefore it is inductively factored, 
by Proposition \ref{prop:superindfactored}.
Thus $\CA_{\ell-1}^{\ell-2}(r) \times \Phi_1$ 
is inductively factored, by Proposition \ref{prop:product-indfactored}.
Let $\pi= (\pi_1,\dots,\pi_{\ell-1})$ 
be an inductive factorization of 
$\CA_{\ell-1}^{\ell-2}(r)$ and set 
$\pi_\ell = \emptyset$. 
Then, with Proposition $\ref{prop:product-indfactored}$, 
$\pi^{(0)}:=(\pi_1,\dots,\pi_\ell)$ 
is an inductive factorization of 
$\CA_0:=\CA_{\ell-1}^{\ell-2}(r) \times \Phi_1$.\\
To obtain $\CA_\ell^{\ell-2}(r)$ from $\CA_0$, 
we have to add each hyperplane of the set 
$\{H_{i,\ell}^z\mid 0\leq i < \ell, 0\leq z < r\}$, 
which contains $(\ell-1)r$ different hyperplanes, to $\CA_0$. 
Comparing the exponents of $\CA_{\ell-1}^{\ell-2}$ 
to those of $\CA_{\ell}^{\ell-2}$, we see that  
$(\ell-1)r-1$ of these 
hyperplanes have to be added to $\pi_\ell$ 
and one to $\pi_{\ell-1}$.
These remaining hyperplanes 
$\{\tilde H_1,\dots,\tilde H_{(\ell-1)r}\}$ 
are ordered as indicated in Table \ref{table0}.
Define $\CA_i:=\CA_{i-1}\cup \{\tilde H_i\}$. 
Furthermore let $\pi^{(i)}$ be the partition 
of $\CA_i$ obtained from $\pi^{(i-1)}$ after $\tilde H_i$ is added. 
The induction of hyperplanes is given in Table \ref{table0} below.
In each step $i=1,\dots,(\ell-1)r$, 
we have to show that $(\CA_i'',(\pi^{(i)})'') \in \CIFAC$ 
as well as 
\[
\exp(\CA_i'')=\{|\pi^{(i)}_1|,\dots,|\pi^{(i)}_{j-1}|,|\pi^{(i)}_{j+1}|,\dots,|\pi^{(i)}_\ell|\}
\text{ with }\tilde H_i \in \pi^{(i)}_j.
\]

\begin{table}[ht!b]\tiny
\renewcommand{\arraystretch}{1.5}
\begin{tabular}{lllll}\hline
$(\pi^{(i)})'$ &  $\exp\CA_i'$ & $\tilde H_i$  & $\exp\CA_i''$\\ 
\hline\hline
$\pi, \{\}$ &  $\exp(\CA_{\ell-1}^{\ell-2}(r)),0$ &  $H_{1,\ell}^{0}$ & $\exp(\CA_{\ell-1}^{\ell-2}(r))$ \\
$\pi, \{H_{1,\ell}^0\}$ &  $\exp(\CA_{\ell-1}^{\ell-2}(r)),1$ &  $H_{1,\ell}^{1}$ & $\exp(\CA_{\ell-1}^{\ell-2}(r))$ \\
$\pi, \{H_{1,\ell}^0, H_{1,\ell}^1\}$ &  $\exp(\CA_{\ell-1}^{\ell-2}(r)),2$ &  $H_{1,\ell}^{2}$ & $\exp(\CA_{\ell-1}^{\ell-2}(r))$ \\
$\vdots$ & $\vdots$ & $\vdots$ & $\vdots$  \\
$\pi, \{H_{1,\ell}^0,\dots,H_{1,\ell}^{r-1}\}$ &  $\exp(\CA_{\ell-1}^{\ell-2}(r)),r$ & $H_{2,\ell}^{0}$ &  $\exp(\CA_{\ell-1}^{\ell-2}(r))$ \\
$\pi, \{H_{1,\ell}^0,\dots,H_{1,\ell}^{r-1},H_{2,\ell}^0\}$ &  $\exp(\CA_{\ell-1}^{\ell-2}(r)),r$ & $H_{2,\ell}^{1}$ &  $\exp(\CA_{\ell-1}^{\ell-2}(r))$ \\
$\vdots$ & $\vdots$ & $\vdots$ & $\vdots$  \\
$\pi, \{H_{1,\ell}^0,\dots,H_{\ell-3,\ell}^{r-1}\}$ &  $\exp(\CA_{\ell-1}^{\ell-2}(r)),(\ell-3)r$ & $H_{\ell-2,\ell}^{0}$ &  $\exp(\CA_{\ell-1}^{\ell-2}(r))$ \\
$\vdots$ & $\vdots$ & $\vdots$ & $\vdots$  \\
$\pi, \{H_{1,\ell}^0,\dots,H_{\ell-2,\ell}^{r-2}\}$ &  $\exp(\CA_{\ell-1}^{\ell-2}(r)),(\ell-2)r-1$ & $H_{\ell-2,\ell}^{r-1}$ &  $\exp(\CA_{\ell-1}^{\ell-2}(r))$ \\
$\pi, \{H_{1,\ell}^0,\dots,H_{\ell-2,\ell}^{r-1}\}$ &  $\exp(\CA_{\ell-1}^{\ell-2}(r)),(\ell-2)r$ & $H_{\ell-1,\ell}^{0}$ &  $\exp(\CA_{\ell-1}^{\ell-2}(r))$ \\
$\pi_1,\dots,\pi_{\ell-1}\cup \{H_{\ell-1,\ell}^0\}, \{H_{1,\ell}^0,\dots,H_{\ell-2,\ell}^{r-1}\}$ &  $\exp(\CA_{\ell-1}^{\ell-1}(r)),(\ell-2)r$ & $H_{\ell-1,\ell}^{1}$ &  $\exp(\CA_{\ell-1}^{\ell-1}(r))$ \\
$\pi_1,\dots,\pi_{\ell-1}\cup \{H_{\ell-1,\ell}^0\}, \{H_{1,\ell}^0,\dots,H_{\ell-2,\ell}^{r-1}, H_{\ell-1,\ell}^1 \}$ &  $\exp(\CA_{\ell-1}^{\ell-1}(r)),(\ell-2)r+1$ & $H_{\ell-1,\ell}^{2}$ &  $\exp(\CA_{\ell-1}^{\ell-1}(r))$ \\
$\vdots$ & $\vdots$ & $\vdots$ & $\vdots$  \\
$\pi_1,\dots,\pi_{\ell-1}\cup \{H_{\ell-1,\ell}^0\}, \{H_{1,\ell}^0, \dots, H_{\ell-1,\ell}^{r-2}\}\setminus \{H_{\ell-1,\ell}^0\}$ &  $\exp(\CA_{\ell-1}^{\ell-1}(r)),(\ell-1)r-2$ & $H_{\ell-1,\ell}^{r-1}$ &  $\exp(\CA_{\ell-1}^{\ell-1}(r))$ \\

$\pi_1,\dots,\pi_{\ell-1}\cup \{H_{\ell-1,\ell}^0\}, \{H_{1,\ell}^0,\dots,H_{\ell-1,\ell}^{r-1}\}\setminus \{H_{\ell-1,\ell}^0\}$ & $\exp(\CA_{\ell}^{\ell-2}(r)) $& \\

\hline
\end{tabular}
\smallskip
\caption{Induction Table of Factorizations for $\CA_\ell^{\ell-2}(r)$}
\label{table0}
\end{table}

To prove that this table is indeed 
an inductive factorization table, 
and therefore that ultimately $\CA_{\ell}^{\ell-2}(r)$ 
is inductively factored, we are going to proceed in 3 steps.

First let $\tilde H_i\in \{H_{1,\ell}^0,\dots,H_{\ell-2,\ell}^{r-1}\}$. 
Since only $\pi^{(i)}_{\ell}$ contains 
hyperplanes that involve the coordinate $x_\ell$, 
one can easily verify that $\CA_i'' \cong \CA_{\ell-1}^{\ell-2}(r)$ and $(\pi^{(i)})'' \cong \pi$, so it is inductively factored by assumption.

For the next step we first need to specify the
the inductive factorization 
$(\pi_1,\dots,\pi_{\ell-1})$ of $\CA_{\ell-1}^{\ell-2}(r)$ 
we intend to start with.
Let $\pi_{\ell-1} = \{H_{j,\ell-1}^n \mid 1\leq j < \ell-1, 0\leq n < r\}$.
This is indeed possible. 
By starting with an arbitrary inductive factorization 
$(\pi_1,\dots,\pi_{\ell-2}, \emptyset)$ for $\CA_{\ell-2}^{\ell-2}\times \Phi_1$ 
and then inductively adding 
the remaining hyperplanes $\{H_{j,\ell-1}^z\mid 1\leq j < \ell-1, 0\leq z < r\}$ 
to the empty part of the partition, 
one can easily verify that this gives an inductive factorization chain.

Now assume that $\tilde H_i=H_{\ell-1,\ell}^0$ and let $\pi_{\ell-1}$ be like before.
It is easy to verify that 
$\pi^{(i-1)}_{\ell} \cong \pi_{\ell-1}$ and that 
$\CA_i'' \cong \CA_{\ell-1}^{\ell-2}(r)$, 
so by assumption we get
\[
\left(\CA_i'',(\pi^{(i)})'')\right) \cong \left(\CA_{\ell-1}^{\ell-2}(r),\pi\right)\in \CIFAC.
\]
Finally let $\tilde H_i\in\{H_{\ell-1,\ell}^1,\dots,H_{\ell-1,\ell}^{r-1}\}$.
Then, since 
$H_{\ell-1,\ell}^0\in \pi^{(i)}_{\ell-1}$ and 
$H_{\ell-1,\ell}^0 \cap \tilde H_i = H_{\ell-1} \cap \tilde H_i$ 
we get a new coordinate hyperplane in the restriction and therefore 
$\CA_i'' \cong \CA_{\ell-1}^{\ell-1}(r)$. \\
Furthermore, we have 
$(\pi^{(i)})''\cong (\pi_1,\dots,\pi_{\ell-2}, \pi_{\ell-1} \cup \{H_{\ell-1}\})$, 
so we have to show that this is indeed an inductive factorization of 
$ \CA_{\ell-1}^{\ell-1}(r)$.\\
Starting with $(\CA_{\ell-1}^{\ell-2}(r),\pi)$ 
and adding $H_{\ell-1}$ to $\pi_{\ell-1}$, 
we get 
\[
\left((\CA_{\ell-1}^{\ell-2}(r))'',\pi''\right) 
= \left(\CA_{\ell-2}^{\ell-2}(r),(\pi_1,\dots,\pi_{\ell-2})\right).
\] 
This is inductively factored due to the previous construction. Consequently,
\[
\left(\CA_i'',(\pi^{(i)})''\right) \cong 
\left(\CA_{\ell-1}^{\ell-1}(r),(\pi_1,\dots,\pi_{\ell-2}, \pi_{\ell-1} \cup \{H_{\ell-1}\})\right) \in \CIFAC
\]
and for $i=(\ell-1)r$ we get that
$\CA_i = \CA_{\ell}^{\ell-2}(r)$
is indeed inductively factored, as desired.
\end{proof}

We are now able to classify all nice and all 
inductively factored instances among the $\CA^k_\ell(r)$.

\begin{theorem}
\label{thm:akl}
Let $\CA = \CA^k_\ell(r)$ for $r,\ell \ge 2$ and $0 \le k \le \ell$.
\begin{itemize}
\item[(i)] 
If $\ell = 2$, then $\CA$ is inductively factored;
\item[(ii)]
If $\ell = 3$, then $\CA$ is nice;
\item[(iii)]
For $\ell \ge 3$, 
$\CA$ is inductively factored if and only if $\ell -2 \le k \le \ell$.
\item[(iv)]
For $\ell \ge 4$, $\CA$ is nice if and only if 
$\CA$ is inductively factored.
\end{itemize}
\end{theorem}

\begin{proof}
If $\ell = 2$, then $\CA$ is supersolvable, so part (i) 
follows from Proposition \ref{prop:superindfactored}.
So suppose that $\ell \ge 3$.
For $k = 0, \ell$ the result follows from 
Theorem \ref{thm:factoredrefl}.
By Theorem \ref{thm:super-restriction}(ii), 
$\CA^{\ell-1}_\ell(r)$ is also supersolvable, 
and so the result follows in this case 
again from Proposition \ref{prop:superindfactored}.

Let $\ell \ge 4$. It follows from 
\cite[Ex.\ 3.8]{moellerroehrle:factored} that 
$\CA^k_\ell(r)$ is not nice for $1 \le k \le \ell-4$.
For completeness, we include the easy argument.
So let $1 \le k \le \ell-4$. Define
\[
X := \bigcap\limits_{\substack{k+1 \leq a < b \leq \ell\\ 0 \leq n < r}} H_{a,b}^n.
\]
Then one checks that 
\[
\CA_X \cong \CA^0_{\ell-k}(r) = \CA(G(r,r,\ell-k)).
\]
For $1 \le k \le \ell-4$, it follows from 
Theorem \ref{thm:factoredrefl}
that $\CA(G(r,r,\ell-k))$ is not nice. Consequently, neither is 
$\CA^k_\ell(r)$, by Remark \ref{rem:factored}.
For $k = \ell-3$, it follows from Lemma \ref{lem:al-3} that 
$\CA^{\ell-3}_\ell(r)$ is not nice either.

Finally, for $r, \ell \ge 3$, Lemma \ref{lem:al-3} shows that
$\CA^{\ell-2}_\ell(r)$ is inductively factored.
This completes the proof of the theorem.
\end{proof}

Theorem \ref{thm:akl} readily extends to the hereditary subclasses.

\begin{corollary}
\label{cor:hered-akl}
Let $\CA = \CA^k_\ell(r)$ for $r,\ell \ge 2$ and $0 \le k \le \ell$.
Then $\CA$ is (inductively) factored if and only if 
$\CA$ is hereditarily (inductively) factored.
\end{corollary}

\begin{proof}
The result follows immediately from Theorem \ref{thm:akl} and the 
pattern of restrictions to hyperplanes among the $\CA^k_\ell(r)$
from \cite[Props.\ 2.11,  2.13]{orliksolomon:unitaryreflectiongroups}
(cf.~\cite[Props.~6.82]{orlikterao:arrangements}).
\end{proof}

\section{Nice Restrictions of Reflection Arrangements}
\label{sec:proof}

Throughout this section, let $W$ be an irreducible unitary reflection group 
and let $\CA = \CA(W)$ be its reflection arrangement.

We begin with a proof of Theorem \ref{thm:nice}.
It follows from Theorem \ref{thm:factoredrefl}(ii) that if 
$\CA$ is nice, then so is $\CA^X$ for any $X \in L(\CA)$.
If $W = G(r,r,\ell)$ for $\ell \ge 3$, then 
the result follows from Theorem \ref{thm:akl}.

This leaves the instances when $W$ is of exceptional type.
If $\dim X \le 2$, then 
$\CA^X$ is supersolvable, and so is nice, by 
Proposition \ref{prop:superindfactored}.
So we concentrate on those instances 
when  
$X \in L(\CA)$ with  $\dim X \ge 3$.
Using the tables
\cite[App.~C, D]{orlikterao:arrangements},
we first address each restriction $\CA^X$
for $W$ of exceptional type and $\dim X = 3$ and some
instances for $\dim X = 4$.
The failure to admit a nice partition is determined computationally.
Then each higher rank restriction can be analyzed by a suitable localization, 
making use of the fact that $L(\CA(W_Y)^X) = L(\CA(W)^X)_Y)$, 
see \cite[Lem.\ 2.11, Cor.\ 6.28]{orlikterao:arrangements}.
It then follows that $\CA(W)^X$ is not nice because the 
localization $(\CA(W)^X)_Y$ fails to be 
nice, cf.~Remark \ref{rem:factored}.
We readily find a 
suitable 3-dimensional localization which is not nice using the tables
\cite[App.\ C]{orlikterao:arrangements}.

The following result is due to Orlik and Terao, 
\cite[App.\ D]{orlikterao:arrangements}.

\begin{lemma}
\label{lem:isoms}
We have the following lattice isomorphisms of 
$3$-dimensional restrictions:
\begin{itemize}
\item [(i)] $(E_6,A_3) \cong \CA_3^2(2)$;
\item [(ii)] $(E_7, D_4) \cong \CA_3^3(2)$;
\item [(iii)] $(F_4,A_1) \cong (F_4,\tilde{A_1}) \cong (E_7,A_1^4) \cong (E_7,(A_1A_3)') \cong (E_8,A_1 D_4) \cong (E_8,D_5)$;
\item [(iv)] $(E_6,A_1 A_2) \cong (E_7,A_4)$;
\item [(v)] $(E_7,A_2^2) \cong (E_8,A_5)$;
\item [(vi)] $(G_{26},A_0) \cong (G_{32}, C(3)) \cong (G_{34}, G(3,3,3))$. 
\end{itemize}
\end{lemma}

We note that the isomorphism between $(E_8,A_1^2 A_3)$ and $(E_8,A_2 A_3)$ claimed in 
\cite[App.\ D]{orlikterao:arrangements} is not correct 
(cf.~\cite{moellerroehrle:zeta}).
We need to treat both restrictions separately.
It follows from Theorem \ref{thm:super-restriction} that the cases listed in 
Lemma \ref{lem:isoms}(i), (ii), and (v) are all supersolvable and thus are 
inductively factored, and so are nice.
Moreover, it follows from Theorem \ref{thm:factoredrefl}(i) 
and \cite[Thm.\ 1.2]{hogeroehrle:super}
that the cases listed in 
Lemma \ref{lem:isoms}(vi) are not nice.
For each of the other kinds it suffices to consider only one 
of the isomorphic restrictions. 
In our next result, we determine all 
nice restrictions of dimension at least $3$
for an ambient  irreducible, non-supersolvable  
reflection arrangement of exceptional type.

\begin{lemma}
\label{lem:dim3}
Let $\CA = \CA(W)$ be an irreducible, non-supersolvable  
reflection arrangement of exceptional type.
Let $X \in L(\CA)$ with $\dim X \ge 3$.
Then $\CA^X$ is nice if and only if 
$\CA^X$ is supersolvable or one of 
$(E_6, A_1A_2) \cong (E_7, A_4), (E_7, (A_1A_3)'')$.
\end{lemma}

\begin{proof}
Using our discussion above together with  
Theorems \ref{thm:factoredrefl} and \ref{thm:super-restriction} along with 
Lemma \ref{lem:isoms}, we still have to analyze the rank 3 restrictions in 
Table \ref{table:rank3restrictions} along with some rank 4 restrictions in 
Table \ref{table:rank4restrictions}.
Using the tables of all orbit types 
for the irreducible reflection groups of exceptional type in 
\cite[App.]{orliksolomon:unitaryreflectiongroups} 
(see also \cite[App.\ C]{orlikterao:arrangements})
as well as the restrictions imposed on factorizations 
in Remark \ref{rem:factored}, 
we determine which of the remaining $3$-dimensional 
restrictions admit a nice partition (Table \ref{table:rank3restrictions})
and also determine that some
4-dimensional restrictions $\CA(W)^X$ 
do not admit a nice partition 
(Table \ref{table:rank4restrictions}).

To illustrate the argument, we indicate this for 
$(E_6,A_1^3)$ which turns out to be not nice. 
The arguments for the other non-nice restrictions in 
Tables \ref{table:rank3restrictions} and 
\ref{table:rank4restrictions} are
quite similar and are left to the reader.
We emphasize that we 
have verified all instances of both tables 
with the aid of a computer by means of the 
computer algebra
system \Sage, \cite{sage}.
The program in question uses 
an intelligent brute force routine which we outline next.
As the singleton part of a partition $\pi$ of $\CA$
it chooses a representative $H$ 
among the orbits of the automorphism group of $L(\CA)$. 
The rest of $\pi$ is then constructed 
as follows. 
Consider the set of all localizations 
$\CA_X$ so that $X\subset H$ and $r(X)=2$. 
For each such $\CA_X$, 
it follows from Corollary \ref{cor:teraofactored}(iii) 
that $\CA_X\setminus \{H\}$
belongs to a single part of $\pi$.
For each such resulting partition $\pi$, the algorithm then tests 
whether the 
conditions in Definition \ref{def:factored} are fulfilled.

Let $\CB$ be the restriction $(E_6, A_1^3)$.
Then we have 
$$\CB=\left\{
\begin{array}{lll}
H_0 = \{ 2x-y-z=0\}, & H_1 = \{ 2x-y+z=0\} ,& H_2 = \{ x+y=0\}, \\
H_3 = \{ 2x+y+z=0\} ,& H_4 = \{ 2x+y-z=0\} ,& H_5 = \{ x-y=0\} ,\\
H_6 = \{ y+z=0\} ,& H_7 = \{ y-z=0\} ,& H_8 = \{ y=0\}, \\
H_9 = \{ x=0\} &&
\end{array} 
\right\}.$$
We claim that $\CB$ does not admit a nice partition.
By way of contradiction, suppose that $\pi = (\pi_1,\pi_2,\pi_3)$ 
is a factorization of $\CB$ and let 
$\exp(\CB) = \{1,4,5\} = \{|\pi_1|,|\pi_2|,|\pi_3|\}$. 
Suppose $X\in L(\CB)$ is a flat of rank 2.
Then $|X| \in \{2,3,4\}$ and one readily checks that  
\begin{itemize}
\item if $|X|=2$, then $X=\{H_i,H_j\}$ with $(i,j) \subset \{0,1,2\}$ or $(i,j) \subset \{3,4,5\}$;
\item if $|X|=3$, then $X=\{H_6,H_7,H_8\}$ or $X=\{H_i,H_j,H_k\}$ with \\ $(i,j,k) \in  \{0,1,2\} \times \{3,4,5\}  \times \{6,7,8\}$
but $(i,j,k) \not\in \{(0,3,6),(1,4,7),(2,5,8)\}$;
\item if $|X|=4$, then $X = \{H_i,H_j,H_k,H_9\}$ with $(i,j,k) \in \{(0,3,6),(1,4,7),(2,5,8)\}$.
\end{itemize}
Now consider a rank 2 flat $X$ with $|X|=4$. Removing 
one hyperplane from $X$ and adding to it a different one 
results in a subset $Z$ of $\CB$ which is not a member of $L(\CB)$.
Then there is always a rank 2 flat $Y$ in $L(\CB)$ satisfying 
$Y \subsetneq Z$. Since every hyperplane of $\CB$ is part of such a set $X$, 
it follows from Corollary \ref{cor:teraofactored} that 
there is no candidate for the singleton part $\pi_1$, 
because the remaining three hyperplanes can't be joined with a fourth 
one to build a part of $\pi$. 
Thus $\pi$ is not a nice partition of $\CB$, 
so $(E_6, A_1^3)$ is not nice.

Next we are going to explicitly determine inductive factorizations
for the two restrictions $(E_6, A_1A_2)$ and $(E_7, (A_1A_3)'')$. 

Let $\CC$ be the restriction $(E_6, A_1A_2)$. 
Then $\exp(\CC)=\{1,4,5\}$ and one checks that 
$$Q(\CC) = x(x\pm z)(x\pm y)(y\pm z)(2x\pm (y-z))(3y-z) \in \BBR[x,y,z].$$
We claim that 
$\pi = (\pi_1, \pi_2, \pi_3)$ with 
\begin{align*}
\pi_1&=\left\{ \{x=0\} \right\},\\
\pi_2&=\left\{
\begin{array}{l}
\{ 2x+y-z =0\}, \{ 2x-y-z=0\}, \\
\{ y-z=0\}, \{ 3y-z=0\} \\
\end{array}
\right\} \text{ and}\\
\pi_3 &= \left\{
\begin{array}{l}
\{ x+z=0\}, \{ x-z=0\},\{ y+z=0\}, \\
\{ x+y=0\}, \{ x-y=0\} \\
\end{array}
\right\}
\end{align*}
is an inductive factorization of $\CC$. 
The niceness of $\pi$ is 
easily verified by checking the 
conditions in Definition \ref{def:factored}.

Let $H_0 = \{3y-z=0\}$, 
then $\CC' = \CC \setminus \{H_0\}$ 
is supersolvable and
$$\{x=0\} \leq \{x=0\} \cap \{ 2x+y-z =0\} \cap\{ 2x-y-z=0\} \cap \{ y-z=0\}  \leq T_{\CC'}$$ 
is a maximal chain of modular elements in $L(\CC')$ 
that induces the factorization $\pi'$ by Proposition \ref{prop:ssfactored},
so $\pi$ indeed is an inductive factorization.

Let $\CD$ be the restriction $(E_7, (A_1A_3)'')$. Then $\exp(\CD)=\{1,5,5\}$ and one checks that
$$Q(\CD)=xyz(x\pm y)(x-z)(x-2z)(y\pm 2z)(x\pm y - 2z) \in \BBR[x,y,z].$$
We claim that $\pi = (\pi_1, \pi_2, \pi_3)$ with 
\begin{align*}
\pi_1&=\left\{ \{x-z=0\} \right\},\\
\pi_2&=\left\{
\begin{array}{l}
\{ x =0\}, \{ z=0\},\{ x-2z=0\} \\
\{ y+2z=0\}, \{ y-2z=0\} \\
\end{array}
\right\} \text{ and}\\
\pi_3 &= \left\{
\begin{array}{l}
\{ y=0\}, \{ x+y-2z=0\},\{ x-y-2z=0\}, \\
\{ x+y=0\}, \{ x-y=0\} \\
\end{array}
\right\}
\end{align*}
is an inductive factorization for $\CD$.
Again, the conditions from 
Definition \ref{def:factored} 
are easy to verify, 
so $\pi$ is a nice partition. 
With $H_0 = \{y-2z=0\}$, 
the map $$\varrho:\CD\setminus \pi_2 \rightarrow \CD'', H\mapsto H\cap H_0$$ 
is a bijection, so with Theorem \ref{thm:add-del-factored}, $(\CD',\pi')$ 
is nice as well. Now let $H_1 = \{y+2z = 0\}$, 
then the arrangement $\CD\setminus \{H_0,H_1\}$ 
is supersolvable and
$$\{x-z=0\} \leq \{x-z=0\}\cap \{ x =0\}\cap \{ z=0\} \cap \{ x-2z=0\} \leq T_{\CD \setminus \{H_0,H_1\}}$$ 
is a maximal chain of modular elements in $L(\CD \setminus \{H_0,H_1\})$
that induces the factorization $(\pi')'$ by Proposition \ref{prop:ssfactored}, 
so $\pi$ is an inductive factorization for $\CD$.

\begin{longtable}{l|l|c}  
\hline
$\CA^X$ & $\exp \CA^X$ & nice \\ 
\hline
\hline
\label{table:rank3restrictions}
$(F_4,A_1)$        &  1,5,7              &false\\
%$(F_4,\tilde A_1)$ &  1,5,7             &false\\
$(G_{29},A_1)$     &  1,9,11         &false\\
$(H_4,A_1)$        &  1,11,19         &false\\
$(G_{31},A_1)$     &  1,13,17         &false\\
%$(G_{32},C(3))$    &  1,7,13            &false\\
$(G_{33},A_1^2)$   &  1,7,9         &false\\
$(G_{33},A_2)$     &  1,6,7           &false\\
$(G_{34},A_1^3)$   &  1,13,19     &false\\
$(G_{34},A_1A_2)$  &  1,13,16     &false\\
$(G_{34},A_3)$     &  1,11,13       &false\\
%$(G_{34},G(3,3,3))$ &  1,7,13         &false\\
$(E_6,A_1^3)$      &  1,4,5              &false\\
$(E_6,A_1A_2)$     &  1,4,5          & {\bf true}\\
%$(E_6,A_3)$        &  1,3,4           & {\bf true}\\
%$(E_7,A_1^4)$      &  1,5,7          &false\\
$(E_7,A_1^2A_2)$   &  1,5,7      &false\\
%$(E_7,A_2^2)$      &  1,5,7      & {\bf true}\\
%$(E_7,(A_1A_3)')$  &  1,5,7       &false\\
$(E_7, (A_1A_3)'')$ &  1,5,5       & {\bf true}\\
%$(E_7,A_4)$        &  1,4,5         & {\bf true}\\
%$(E_7,D_4)$        &  1,3,5          & {\bf true}\\
$(E_8,A_1^3A_2)$   &  1,7,11    &false\\
$(E_8,A_1A_2^2)$   &  1,7,11      &false\\
$(E_8,A_1^2A_3)$   &  1,7,9        &false\\
$(E_8,A_2A_3)$     &  1,7,9          &false\\
$(E_8,A_1A_4)$     &  1,7,8          &false\\
%$(E_8,A_5)$        &  1,5,7           & {\bf true}\\
%$(E_8,A_1D_4)$     &  1,5,7          &false\\
%$(E_8,D_5)$        &  1,5,7            &false\\
\hline
\caption{Rank 3 restrictions of the exceptional groups}
\end{longtable}

\begin{longtable}{l|l|c}  
\hline
$\CA^X$ & $\exp \CA^X$ & nice \\ %\endfirsthead
%\endhead\toprule 
\hline
\hline
\label{table:rank4restrictions}
$(G_{33},A_1)$  &  1,7,9,11    &false\\
$(E_6,A_1^2)$      &  1,4,5,7       &false\\
$(E_6,A_2)$     &  1,4,5,5      & false\\
$(E_7,A_1A_2)$   &  1,5,7,8      &false\\
$(E_7,A_3)$   &  1,5,5,7  &false\\
$(E_8, A_2^2)$  &  1,7,11,11    &false\\
$(E_8,A_4)$  & 1,7,8,9  &false\\
$(E_8,D_4)$  & 1,5,7,11    &false\\
\hline
\caption{Some rank 4 restrictions of exceptional groups}
\end{longtable}

For the higher rank restriction, 
we utilize a suitable localization of rank $3$ from 
Table \ref{table:rank3restrictions}, of rank $4$ from 
Table \ref{table:rank4restrictions},   or else 
one which appears earlier in the same table and 
which is not nice.

\begin{longtable}{l|l}  
\hline
$\CA^X$ & non-nice localization \\ %\endfirsthead
%\endhead\toprule 
\hline
\hline
\label{table:higherrankrestrictions}
$(G_{34},A_1)$ & $(D_5,A_1) = A_4^1(2)$ \\
$(G_{34},A_1^2)$ & $(G_{33},A_1^2)$ \\
$(G_{34},A_2)$ & $(G_{33},A_2)$\\
\hline
$(E_6,A_1)$ & $(D_5,A_1) = A_4^1(2)$ \\
\hline
$(E_7,A_1)$ & $(D_5,A_1) = A_4^1(2)$ \\
$(E_7,A_1^2)$ & $(E_6,A_1^2)$ \\
$(E_7,A_2)$ & $(E_6,A_2)$\\
$(E_7,(A_1^3)')$ & $(E_6,A_1^3)$\\
$(E_7,(A_1^3)'')$ & $(E_6,A_1^3)$\\
\hline
$(E_8,A_1)$ &  $(D_5,A_1) = A_4^1(2)$ \\
$(E_8,A_1^2)$ & $(E_7,A_1^2)$\\
$(E_8,A_2)$ & $(E_7,A_2)$\\
$(E_8,A_1^3)$ & $(E_6,A_1^3)$\\
$(E_8,A_1 A_2)$ & $(E_7,A_1 A_2)$\\
$(E_8,A_3)$ & $(E_7,A_3)$\\
$(E_8,A_1^4)$ & $(E_7,A_1^4)$ \\
$(E_8,A_1^2 A_2)$ & $(E_7,A_1^2A_2)$\\
$(E_8,A_1 A_3)$ & $(E_7,(A_1A_3)')$\\
\hline
\caption{Higher rank restrictions in exceptional groups}
\end{longtable}

The result now follows from Table \ref{table:higherrankrestrictions},  
Remark \ref{rem:factored}, the data in Table \ref{table:rank3restrictions}, 
as well as Lemma \ref{lem:isoms} and our discussion above.
\end{proof}

Next we prove Theorem \ref{thm:indfac}.
It follows from Theorem \ref{thm:factoredrefl}(ii) that if 
$\CA$ is inductively factored, then so is $\CA^X$ for any $X \in L(\CA)$.
If $W = G(r,r,\ell)$ for $\ell \ge 3$, then 
the result follows from Theorem \ref{thm:akl}.

It follows from Theorem \ref{thm:indfree-restriction}
and Proposition \ref{prop:indfactoredindfree} that 
for $\dim X \ge 4$, $\CA^X$ is not inductively factored.

If $W$ is exceptional, then the result follows from 
the next lemma. Let $\CA = \CA(W)$.
Recall that if $\dim X \le 2$, then  
$\CA^X$ is supersolvable and so is inductively factored.
On the other hand, by 
Proposition \ref{prop:superindfactored}
and Theorem \ref{thm:indfree-restriction}(iii),
the restricted arrangement $\CA^X$ is 
not inductively factored, 
for $X \in L(\CA)$ with $\dim X \ge 4$.

\begin{lemma}
\label{lem:exeptional-indfac}
Let $W$ be irreducible of exceptional type with reflection 
arrangement $\CA = \CA(W)$. Let $X \in L(\CA)$.
Then $\CA^X$ is nice if and only if it is inductively factored.
\end{lemma}

\begin{proof}
It follows from Lemma \ref{lem:dim3} that up to isomorphism
there are only two non-supersolvable, nice restrictions.
One checks that the calculated nice partitions
are inductively factored. 
\end{proof}

Finally, Theorem \ref{thm:hered-indfac}
follows from Theorems \ref{thm:indfac}, \ref{thm:akl},
along with \cite[Prop.\ 6.82]{orlikterao:arrangements}
and Lemma \ref{lem:3-arr}.

%%%%%%%%%%%%%%%%%%%%%%%%%%%%%%%%%%%%%%%%%%%%%%%%%%%%%%%%%%%%%%%%%%%%%%
%%%%%%%%%%%%% Acknowledgments
%%%%%%%%%%%%%%%%%%%%%%%%%%%%%%%%%%%%%%%%%%%%%%%%%%%%%%%%%%%%%%%%%%%%%%
\medskip
{\bf Acknowledgments}:
We acknowledge 
support from the DFG-priority program 
SPP1489 ``Algorithmic and Experimental Methods in
Algebra, Geometry, and Number Theory''.

%%%%%%%%%%%%%%%%%%%%%%%%%%%%%%%%%%%%%%%%%%%%%%%%%%%%%%%%%%%%%%%%%%%%%%
%%%%%%%%%%%%% bibliography
%%%%%%%%%%%%%%%%%%%%%%%%%%%%%%%%%%%%%%%%%%%%%%%%%%%%%%%%%%%%%%%%%%%%%%

\bigskip

\bibliographystyle{amsalpha}

\newcommand{\etalchar}[1]{$^{#1}$}
\providecommand{\bysame}{\leavevmode\hbox to3em{\hrulefill}\thinspace}
\providecommand{\MR}{\relax\ifhmode\unskip\space\fi MR }
% \MRhref is called by the amsart/book/proc definition of \MR.
\providecommand{\MRhref}[2]{%
  \href{http://www.ams.org/mathscinet-getitem?mr=#1}{#2} }
\providecommand{\href}[2]{#2}

%%%%%%%%%%%%%%%%%%%%%%%%%%%%%%%%%%%%%%%%%%%%%%%%%%%%%%%%%%%%%%%%%%%%%%
%%%%%%%%%%%%%%%%%%%%%%%%%%%%%%%%%%%%%%%%%%%%%%%%%%%%%%%%%%%%%%%%%%%%%%

\end{document}